\documentclass{amsproc}

\usepackage{amsmath}
\usepackage{amssymb}
\usepackage{rotating}
\usepackage{graphicx}
\usepackage[tableposition=below]{caption}
\usepackage{bm}

\DeclareGraphicsExtensions{.png,.pdf,.eps}
\usepackage[all,2cell,cmtip]{xy}
\usepackage{tikz}
\usepackage{color}
\usepackage{longtable}
\usepackage{soul}
\newtheorem{theorem}{Theorem}[section]
\newtheorem{lemma}[theorem]{Lemma}
\newtheorem{corollary}[theorem]{Corollary}
\newtheorem{proposition}[theorem]{Proposition}
\newtheorem{conjecture}[theorem]{Conjecture}

\theoremstyle{definition}

\newtheorem{definition}[theorem]{Definition}
\newtheorem{notation}[theorem]{Notation}

\newtheorem{remark}[theorem]{Remark}

\newtheorem{example}[theorem]{Example}
\newtheorem{question}[theorem]{Question}

\def\C{{\mathbb C}}

\def\P{{\mathbb P}}
\def\Q{{\mathbb Q}}
\def\R{{\mathbb R}}
\def\Z{{\mathbb Z}}

\def\a{{\bf a}}

\def\cO{{\mathcal{O}}}

\def\cS{{\mathcal S}}

\def\Q{{\mathbb{Q}}}

\def\d{{\bf d}}

\def\operatorname#1{\mathop{\rm #1}\nolimits}

\def\codim{\operatorname{codim}}

\def\deg{\operatorname{deg}}

\def\NE{{\operatorname{NE}}}
\def\ME{{\operatorname{ME}}}
\def\Nef{{\operatorname{Nef}}}

\def\Eff{{\operatorname{Eff}}}
\def\Mov{{\operatorname{Mov}}}
\newcommand{\cNE}[1]{\overline{\NE}}
\newcommand{\cME}[1]{\overline{\ME}}
\newcommand{\cEff}{\overline{\Eff}}
\newcommand{\cMov}{\overline{\Mov}}

\def\Sp{\operatorname{Sp}}

\makeindex

\begin{document}

\title{Varieties with nef diagonal}

\author{Taku Suzuki}
\address{Faculity of Education, Utsunomiya University, 350 Mine-machi, Utsunomiya, Tochigi 321-8505, Japan} 
\email{taku.suzuki@cc.utsunomiya-u.ac.jp}
\thanks{}
\author{Kiwamu Watanabe}
\date{\today}
\address{Course of Matheorematics, Programs in Matheorematics, Electronics and Informatics, 
Graduate School of Science and Engineering, Saitama University.
Shimo-Okubo 255, Sakura-ku Saitama-shi, 338-8570, Japan.}
\email{kwatanab@rimath.saitama-u.ac.jp}
\thanks{The second author is partially supported by JSPS KAKENHI Grant Number 17K14153.}
\subjclass[2010]{14C25, 14J40, 14J45, 14M10, 14M17.}
\keywords{Nef diagonal, Nef tangent bundle, Algebraic cycles}

\begin{abstract}  For a smooth projective variety $X$, we consider when the diagonal $\Delta_X$ is nef as a cycle on $X\times X$. In particular, we give a classification of complete intersections and smooth del Pezzo varieties where the diagonal is nef. We also study the nefness of the diagonal for spherical varieties. 
\end{abstract}

\maketitle

\section{Introduction} 
Any smooth projective variety $X$ comes equipped with the tangent bundle $T_X$. So the tangent bundle is often used for classification problems of algebraic varieties. 
One of the most important results in this direction is Mori's solution of the Hartshorne-Frankel conjecture characterizing projective spaces as the only smooth projective varieties with ample tangent bundle \cite{Mori1}. As a generalization of Mori's result, F. Campana and T. Peternell conjectured that the only complex smooth Fano varieties with nef tangent bundle are rational homogeneous \cite{CP1}. We call it the {\it CP-conjecture} for short. The CP-conjecture has been varified up to dimension five and for certain special classes of varieties. We refer the reader to \cite{MOSWW} and \cite{Kane2}. 

Based on the fact $T_X$ is the normal bundle of the diagonal $\Delta_X$ in the product $X\times X$, B. Lehmann and C. Ottem recently studied how the geometry of $X$ is reflected in the positivity properties of $\Delta_X$ in their paper \cite{LO}. For example, if $T_X$ is nef, then so is the diagonal $\Delta_X$ as a cycle on $X\times X$. However in general the converse is not true. For instance, any fake projective space, which is a smooth projective variety with the same Betti numbers as $\P^n$ but not isomorphic to $\P^n$, has nef diagonal. This means that the class of varieties with nef diagonal is strictly larger than that of varieties with nef tangent bundle. The nefness of the diagonal imposes restrictions on the structure of varieties. For instance, if $X$ is a projective variety with nef diagonal, then every (possibly higher-codimensional) pseudoeffective cycle is nef (see \cite[Theorem 1.4, Proposition~4.1]{LO}). This yields that every extremal contraction of $X$ is of fiber type (Proposition~\ref{prop:bir}). Moreover, if the diagonal $\Delta_X$ is nef and big as a cycle on $X\times X$, then the Picard number of $X$ is one (Proposition~\ref{prop:nefbig:Pic}).

One of the purposes of this paper is to study the nefness of the diagonal for complete intersections of hypersurfaces and smooth del Pezzo varieties:     

\begin{theorem}[Theorem~\ref{them:ci}, Theorem~\ref{thm:dPnef}]\label{MT} 
Let $X$ be a smooth complex projective variety. Assume that the diagonal $\Delta_X$ is nef as a cycle on $X \times X$. Then the following holds.
\begin{enumerate}
\item If $X$ is a complete intersection of hypersurfaces, then $X$ is a projective space, a quadric or an odd-dimensional complete intersection of two quadrics.
\item If $X$ is a del Pezzo variety, then $X$ is one of the following:
\begin{enumerate}
\item an odd-dimensional complete intersection of two quadrics,
\item the Grassmannian $G(2, \C^5)$,
\item a $3$-dimensional linear section of the Grassmannian $G(2, \C^5) \subset \P^9$ embedded via the Pl\"ucker embedding,
\item $\P^1\times \P^1 \times \P^1$,
\item $\P^2 \times \P^2$ or
\item $\P(T_{\P^2})$.
\end{enumerate}
\end{enumerate}
\end{theorem} 
We know all varieties in the above theorem have nef diagonal except odd-dimensional complete intersections of two quadrics (see Remark~\ref{rem:22} and Question~\ref{ques:22}). Moreover this theorem yields the following:

\begin{corollary}[Corollary~\ref{cor:ci:big}, Corollary~\ref{cor:dp:big}] Let $X$ be a smooth complex projective variety. Assume that the diagonal $\Delta_X$ is nef and big as a cycle on $X \times X$. Then the following holds.
\begin{enumerate}
\item $X$ is a complete intersection of hypersurfaces if and only if $X$ is a projective space or an odd-dimensional quadric.
\item $X$ is a del Pezzo variety if and only if $X$ is a $3$-dimensional linear section of the Grassmannian $G(2, \C^5) \subset \P^9$.
\end{enumerate}
In both cases, $X$ is a projective space or a fake projective space.
\end{corollary}

As a byproduct of Theorem~\ref{MT} in a different direction, we may conclude that the CP-conjecture holds for complete intersections of hypersurfaces (Corollary~\ref{cor:CPci}) and del Pezzo varieties (Corollary~\ref{cor:CP}). For complete intersections, it was firstly proved by R. Pandharipande \cite{Pan} (see also \cite{Fur}). On the other hand, it is easy to check that the CP-conjecture holds for almost all smooth del Pezzo varieties. However, to the best of our knowledge, there is no literature stating the CP-conjecture holds for del Pezzo varieties of degree one and two.

We also study spherical varieties with nef diagonal. The following directly follows from \cite{Li}:
\begin{theorem}[Proposition~\ref{prop:Sph}]\label{them:spherical} For any smooth projective spherical variety $X$, the following are equivalent to each other.
\begin{enumerate}
\item The diagonal $\Delta_X$ is nef.
\item For any $k$, the pseudoeffective cone of $k$-cycles coincides with the nef cone of $k$-cycles.
\end{enumerate}
\end{theorem} 

As an example of spherical varieties, we prove the odd symplectic Grassmannian of lines has non-nef diagonal provided that it is not homogeneous (Proposition~\ref{prop:OSGL}). We also give explicit descriptions of the nef cone and the pseudoeffective cone of cycles for a five-dimensional odd symplectic Grassmannian in Proposition~\ref{prop:dp:cone}. The result gives an answer to a question on cones of cycles of horospherical varieties raised by Q. Li (Remark~\ref{rem:Li}).

\section{Preliminaries}

Along this paper, we work over the complex number field. We employ the notation as in \cite{Har}. For a smooth complex projective variety $X$, we denote by $b_i(X)$ the $i$-th Betti number of the complex manifold $X^{an}$ associated to $X$. 

\subsection{Cycles and cones}
For the intersection theory, we refer the reader to \cite{Ful}. Let $X$ be a smooth projective variety of dimension $n$.  A {\it $k$-cycle} on $X$ is a finite formal linear combination $\sum a_i [Z_i]$ where $Z_i$ are $k$-dimensional closed subvarieties of $X$ and $a_i$ are real numbers. The cycle $\sum a_i [Z_i]$ is {\it effective} if $a_i\geq 0$ for any $i$. The cycle $\sum a_i [Z_i]$ is {\it nef} if it has non-negative intersection numbers with all $(n-k)$-dimensional closed subvarieties. We denote by $A_k(X)=A^{n-k}(X)$ (resp. $N_k(X)=N^{n-k}(X)$) the group of $k$-cycles with real coefficients modulo rational equivalence (resp. numerical equivalence) on $X$. We may take the intersection product of cycles on the Chow ring $A(X):=\bigoplus A_k(X)$ or $N(X):=\bigoplus N_k(X)$. Since numerical equivalence is coarser than rational equivalence, we have a surjective ring homomorphism $A(X)\to N(X)$.
By \cite[Example~19.1.4]{Ful}, $N_k(X)$ is a finite-dimensional $\R$-vector space. We may define the degree homomorphism $\deg: {A}_0(X)\to \R$ (see \cite[Definition~1.4]{Ful}). For a zero-cycle $\sum a_i [P_i]$ ($P_i \in X$), $\deg \sum a_i [P_i]=\sum a_i$. For cycles $Z_i \in A^{k_i}(X)$ $(\sum_i^r k_i=n)$, their intersection product $Z_1 \cdot Z_2 \cdots Z_r$ is a zero-cycle. Then the intersection number of $Z_i$'s is the degree $\deg (Z_1 \cdot Z_2 \cdots Z_r)$. 

The {\it pseudoeffective cone} $\overline{{\rm Eff}}_k(X)=\overline{{\rm Eff}}^{n-k}(X)$ in $N_k(X)$ is the closure of the cone generated by effective $k$-cycles on $X$. A cycle in the pseudoeffective cone is called a {\it pseudoeffective cycle}. A {\it $k$-cycle} $Z$ on $X$ is {\it big} if $Z$ lies in the interior of $\overline{{\rm Eff}}_k(X)$. A {\it $k$-cycle} $Z$ on $X$ is {\it universally pseudoeffective} if $\pi^{\ast}Z$ is pseudoeffective for any morphism of projective varieties $f: Y \to X$ with $Y$ smooth (see \cite[Section~1.3]{FL2}). The {\it nef cone} ${\rm Nef}_k(X)={\rm Nef}^{n-k}(X)$ in $N_k(X)$ is the cone generated by nef $k$-cycles on $X$. Remark that the nef cone ${\rm Nef}_k(X)$ is the dual of the pseudoeffective cone $\overline{{\rm Eff}}^{k}(X)$ via the intersection pairing.

Let $E$ be a vector bundle of rank $r$ on $X$. For each $k=0, 1, \ldots, r$, the {\it $k$-th Chern class} $c_k(E) \in A^k(X)$ is defined by the relations
$$\sum_{k=0}^r(-1)^k\pi^{\ast}c_k(E)\xi^{r-k}=0~\text{and}~c_0(E)=1,
$$
where $\xi$ is the divisor associated to the tautological line bundle $\cO_{\P(E)}(1)$ and $\pi: \P(E) \to X$ is the natural projection. The {\it Chern polynomial} of $E$ is defined by
$$
c_t(E):=c_0(E)+c_1(E)t+ \cdots+ c_r(E)t^r.
$$

\subsection{Properties of varieties with nef diagonal}\label{subsec:nef}

For a smooth projective variety $X$, the diagonal $\Delta_X \subset X \times X$ is said to be {\it nef} (resp. {\it big}) if $\Delta_X$ is nef (resp. {big}) as a cycle on $X \times X$.

\begin{proposition}[{\cite[Theorem~1.4, Proposition~4.1]{LO}}]\label{prop:pseffcone} Let $X$ be a smooth projective variety with nef diagonal. Then every pseudoeffective class on $X$ is nef.
\end{proposition}

\begin{proposition}[{\cite[Proposition~8.1.12]{Ful}}]\label{prop:diag^2} For a smooth projective variety $X$, let $\Delta_X \subset X \times X$ be the diagonal of $X$. Then 
$$ \deg \Delta_X^2= \deg c_n(X).
$$ 
\end{proposition}

\begin{remark}\label{rem:euler}
Let $X$ be a smooth complex projective variety. By the Gauss-Bonnet theorem (see for instance \cite[P. 70, Chapter~I, Theorem~4.10.1]{Hir}), the degree of the top Chern class of $X$ is nothing but the topological Euler characteristic: $$\deg c_n(X)=\sum_{i} (-1)^{i}b_i(X).$$
\end{remark}


\begin{example}\label{ex:curve} Let $C$ be a smooth projective curve with nef diagonal. By Proposition~\ref{prop:diag^2}, $\deg c_1(C)=\deg \Delta_C^2 \geq 0$. This implies that the only smooth projective curves with nef diagonal are a projective line and elliptic curves.
\end{example}

\begin{proposition}[{\cite[Proposition~4.4]{LO}}]\label{prop:LO} Let $X$ be a smooth projective variety of dimension $n$ admitting a finite morphism $f :X \to \P^n$  of degree $d$. Suppose that $\deg c_n(X)>(n+1)d$. Then $\Delta_X$ is not nef.
\end{proposition}

\begin{proposition}\label{prop:bir} Let $X$ be a smooth projective variety with nef diagonal. Then any extremal contraction of $X$ is of fiber type.
\end{proposition}

\begin{proof} Let $f: X \to Y$ be an extremal contraction of birational type. By Proposition~\ref{prop:pseffcone}, a projective curve $C \subset X$ contracted by $f$ is a nef $1$-cycle. For an ample divisor $H$ on $Y$, $f^{\ast}H$ is nef and big. Thus there exist an ample $\Q$-divisor $A$ and an effective $\Q$-divisor $E$ such that $f^{\ast}H=A+E \in N^1(X)$. Then we have $(A+E).C\geq A.C >0$, but this contradicts to $f^{\ast}H.C=0$. 

\end{proof}

\begin{proposition}[{\cite[Corollary~4.2.(1)]{LO}}]\label{prop:nefbig:Pic} Let $X$ be a smooth projective variety with nef and big diagonal. Then $N^1(X) \cong \R$.
\end{proposition}

\begin{definition} A {\it fake projective space} is a smooth projective variety with the same Betti numbers as $\P^n$ but not isomorphic to $\P^n$.
\end{definition}

\begin{proposition}[{\cite[Section~1]{LO}}]\label{prop:fps} A projective space and any fake projective space have nef and big diagonal. 
\end{proposition}

\section{Complete intersections}\label{sect:ci}



In this section, we will study the case of complete intersections. The following two results are well known:

\begin{proposition}[{\cite[Theorem~3]{Az}}]\label{prop:euler} Let $X$ be an $n$-dimensional smooth projective complete intersection of type $(d_1, d_2, \ldots, d_r)$. Then the degree of the top Chern class $c_n(X)$ is given by
$$\deg c_n(X)=d_1d_2\cdots d_r\Biggl\{\sum_{i=0}^n  (-1)^{n-i} \binom{n+r+1}{i} h_{n-i}(d_1,\ldots,d_r)\Biggr\},
$$
where $\displaystyle{h_k(d_1,\ldots,d_r):=\sum_{\substack{i_1+i_2+ \ldots + i_r=k\\ i_j \geq 0}} d_1^{i_1}d_2^{i_2}\cdots d_r^{i_r}}$.
\end{proposition}

\begin{definition}\label{def:eq} For a set of non-negative integers $(d_1, \ldots, d_r, n)$, we put
$$\chi(d_1, d_2, \ldots, d_r; n)=d_1d_2\cdots d_r\Biggl\{\sum_{i=0}^n  (-1)^{n-i} \binom{n+r+1}{i} h_{n-i}(d_1,\ldots,d_r)\Biggr\}
$$ 
\end{definition}

\begin{proposition}[{\cite[P. 146]{Az}}]\label{cor:formula} \rm Under the notation as in Definition~\ref{def:eq}, we have a formula
$$
\chi(d_1, d_2, \ldots, d_r; n)=d_1\chi(d_2, \ldots, d_r; n)-(d_1-1)\chi(d_1, d_2, \ldots, d_r; n-1)
$$
\end{proposition}

\begin{proposition}\label{prop:hyp} Let $X$ be an $n$-dimensional smooth projective hypersurface of degree $d \geq 3$. If $n$ is even, then the degree of the top Chern class $\deg c_n(X)$ is positive. If $n$ is odd, then $\deg c_n(X)$ is negative except the case where $(n, d)=(1,3)$.
\end{proposition}

\begin{proof} By Proposition~\ref{prop:euler}, we have an equality
$$d \deg c_n(X)=(1-d)^{n+2}-1+d(n+2).
$$ Then our claim follows from this equality.
\end{proof}

\begin{proposition}\label{prop:ci} Let $X$ be an $n$-dimensional smooth projective complete intersection of type $(d_1, d_2, \ldots, d_r)$. Assume that $2 \leq r$, $2\leq d_1 \leq d_2 \leq \ldots \leq d_r$ and $3 \leq d_r$. Then, if $n$ is even, then the degree of the top Chern class $\deg c_n(X)$ is positive. If $n$ is odd, then $\deg c_n(X)$ is negative.
\end{proposition}

\begin{proof} Let us first consider the case where $n=1$. In this case, by Proposition~\ref{cor:formula}, we have an equality
\begin{eqnarray}\label{eq:chi} 
\chi(d_1, d_2, \ldots, d_r; 1)&=&d_1\chi(d_2, \ldots, d_r; 1)-(d_1-1)\chi(d_1, d_2, \ldots, d_r; 0) \\
&=&d_1\chi(d_2, \ldots, d_r; 1)-(d_1-1)d_1d_2\cdots d_r. \nonumber
\end{eqnarray}
When $r=2$, we have $\chi(d_1,d_2;1)=d_1\chi(d_2;1)-(d_1-1)d_1d_2$. Since it follows from Proposition~\ref{prop:hyp} that $\chi(d_2;1)$ is nonpositive, $\chi(d_1,d_2;1)$ is negative. By induction on $r$, the equality (\ref{eq:chi}) tells us that $\chi(d_1, d_2, \ldots, d_r; 1)$ is negative. Hence our assertion holds for $n=1$ and $r \geq 2$. 

To prove our statement, we use induction on $n+r$. Remark that we have already dealt with the case where $n \geq 2$ and $r=1$ in Proposition~\ref{prop:hyp}. Hence if $n+r=3$, then our assertion holds. Put $m:=n+r$ and suppose the result is true for $m-1\geq 3$. By Proposition~\ref{cor:formula}, we have an equation 
$$\chi(d_1, d_2, \ldots, d_r; n)=d_1\chi(d_2, \ldots, d_r; n)-(d_1-1)\chi(d_1, d_2, \ldots, d_r; n-1)
$$
By the induction hypothesis, if $n$ is even, then we see that $\chi(d_2, \ldots, d_r; n)>0$ and $\chi(d_1, d_2, \ldots, d_r; n-1)<0$. Thus $\deg c_n(X)=\chi(d_1, d_2, \ldots, d_r; n)$ is positive. If $n$ is odd, the same argument implies the negativity of $\deg c_n(X)=\chi(d_1, d_2, \ldots, d_r; n)$.
\end{proof}

\begin{corollary}\label{cor:even} Let $X$ be an $n$-dimensional smooth projective complete intersection of type $(d_1, d_2, \ldots, d_r)$. Assume one of the following holds:
\begin{enumerate}
\item $r=1$ and $(n,d_1)\neq (2,3)$, or
\item $2 \leq r$, $2 \leq d_1 \leq d_2 \leq \ldots \leq d_r$ and $3 \leq d_r$.
\end{enumerate}
If $n$ is even, then $\deg c_n(X)>(n+1)d_1d_2\cdots d_r$.
\end{corollary}

\begin{proof} As the first case, assume that $r=1$ and $(n,d_1)\neq (2,3)$. By Proposition~\ref{prop:euler}, we have an equality
$$\deg c_n(X)=\frac{(1-d_1)^{n+2}-1}{d_1}+n+2.
$$ 
Then it is straightforward to show that $\deg c_n(X)>(n+1)d_1$. We also remark that $\deg c_n(X)=(n+1)d_1$ provided that $r=1$ and $(n,d_1)= (2,3)$.

As the second case, let us assume that $2 \leq r$, $2 \leq d_1 \leq d_2 \leq \ldots \leq d_r$ and $3 \leq d_r$. We proceed induction on $r$. When $r=2$, the previous argument implies that $\chi(d_2;n) \geq (n+1)d_2$, and it follows from Proposition~\ref{prop:ci} that $\chi(d_1, d_2; n-1) <0$. Hence, by Proposition~\ref{cor:formula}, 
$$\chi(d_1, d_2; n)=d_1\chi(d_2; n)-(d_1-1)\chi(d_1, d_2; n-1)>(n+1)d_1d_2.
$$ We may therefore assume that our claim holds for $r-1$. Then we have $$\chi(d_2, \ldots, d_r; n)>(n+1)d_2\cdots d_r,$$ and it follows from Proposition~\ref{prop:ci} that $\chi(d_1, d_2, \ldots, d_r; n-1)<0$. Thus Proposition~\ref{cor:formula} concludes 
$$
\chi(d_1, d_2, \ldots, d_r; n)>(n+1)d_1\cdots d_r
$$
\end{proof}

\begin{proposition}\label{prop:quad} Let $X$ be an $n$-dimensional smooth projective complete intersection of $r$ quadrics. Assume that $r \geq 3$. If $n$ is odd, then $\deg c_n(X)$ is negative. If $n$ is even and $(n, r) \neq (2,3)$, then $\deg c_n(X)>2^r(n+1)$.
\end{proposition}

\begin{proof} 
For positive integers $n,r$, let us set
$$
b(n,r):=\frac{(-1)^n \deg c_n(X)}{2^r}.
$$ 
By Proposition~\ref{prop:euler} and Proposition~\ref{cor:formula}, we have a relation
\begin{equation}\label{eq:b}
b(n,r)=b(n,r-1)+b(n-1,r).
\end{equation}
It also follows from Proposition~\ref{prop:euler} that 
$$b(n,1)=\frac{(-1)^{n}(2n+3)+1}{4}~~\mbox{and}~~ b(1,r)=r-2.$$  
By a straightforward computation, we have 
\begin{equation}\label{eq:b2}
b(n,2)=
\begin{cases}
\frac{n}{2}+1 & \text{if $n$ is even} \\
0 & \text{if $n$ is odd}
\end{cases}
\end{equation}
To prove the proposition, it is enough to show the following claim:
\begin{enumerate}
\item[(i)] If $r \geq 3$, then $b(n,r)>0$.
\item[(ii)] If $r \geq 3$, $n$ is even and $(n,r)\neq (2,3)$, then $b(n,r)>n+1$. 
\end{enumerate}
From now on, we assume that $r \geq 3$.  By the above equation (\ref{eq:b2}), $b(n,2)$ is non-negative. Thus it follows from (\ref{eq:b}) that 
$$b(n,3)=b(n,2)+b(n-1,3)\geq b(n-1,3) \geq \ldots \geq b(1,3)=1.$$
Since we also see $b(1,r)=r-2\geq 1$, the above equation (\ref{eq:b}) concludes $\rm (i)$.

Assume that $n \geq 2$ and $r \geq 3$. By (\ref{eq:b}) and $\rm (i)$, we have 
$$b(n,r)=b(n,r-1)+b(n-1,r) >b(n, r-1).$$
If $n=2$, then for any $r \geq 4$, we have
$$
b(2,r)>b(2,3)=3=n+1.
$$
If $n=4$, then for any $r \geq 3$, we have
$$
b(4,r)\geq b(4,3)=6>n+1.
$$
By induction on $n$, if $n$ is even such that $n \geq 6$, then for any $r \geq 3$,
\begin{eqnarray}
b(n,r)&\geq& b(n,3)=b(n,2)+b(n-1,3) \nonumber \\
&=&\frac{n}{2}+1+b(n-1,2)+b(n-2,3) \nonumber \\
&>& \frac{n}{2}+1 +(n-1)=\frac{3}{2}n>n+1. \nonumber 
\end{eqnarray}
\end{proof}

\begin{theorem}\label{them:ci} 
Let $X$ be a smooth projective complete intersection of hypersurfaces. Assume that the diagonal $\Delta_X$ is nef. Then $X$ is a projective space, a quadric or an odd-dimensional complete intersection of two quadrics.
\end{theorem}

\begin{proof} Let $X$ be an $n$-dimensional smooth projective complete intersection of type $(d_1, d_2, \ldots, d_r)$. Assume moreover $X$ is not a projective space, a quadric or an odd-dimensional complete intersection of two quadrics.
We may assume that $n \geq 2$ and $d_i>1$. By Lemma~\ref{lem:except} below, we may also assume that $X$ is not a cubic surface, a $2$-dimensional complete intersection of type $(2,2, 2)$ and an even-dimensional complete intersection of type $(2,2)$.   
Then, applying Proposition~\ref{prop:ci}, Corollary~\ref{cor:even} and Proposition~\ref{prop:quad}, we see that 
$$\deg c_n(X)<0\,\,\,\,\mbox{or}\,\, \deg c_n(X)>d_1d_2 \cdots d_r(n+1).$$ 
If $\deg c_n(X)<0$, then Proposition~\ref{prop:diag^2} tells us $\deg \Delta_X^2=\deg c_n(X)<0$. Hence $\Delta_X$ is not nef. Suppose that $\deg c_n(X)>d_1d_2 \cdots d_r(n+1)$. For $X \subset \P^{n+r}$, let us consider a general projection $\pi: X \to \P^n$ from a linear subspace of dimension $r-1$. Then $\pi$ is a finite morphism of degree $d_1d_2 \cdots d_r$. Applying Proposition~\ref{prop:LO}, we see that $\Delta_X$ is not nef.
\end{proof}

\begin{lemma}\label{lem:except} Let $X$ be one of the following:
\begin{enumerate}
\item a cubic surface,
\item a $2$-dimensional complete intersection of type $(2, 2, 2)$ or
\item an even-dimensional complete intersection of type $(2,2)$.
\end{enumerate}
Then $\Delta_X$ is not nef
\end{lemma}

\begin{proof} Any cubic surface has a $(-1)$-curve. Then it follows from Proposition~\ref{prop:pseffcone} that the diagonal of a cubic surface is not nef. If $X$ is a $2$-dimensional complete intersection of type $(2,2, 2)$, then it is a K3 surface. Hence $\Delta_X$ is not nef by \cite[Theorem~6.6]{LO}. 
Let $X$ be an $2n$-dimensional complete intersection of type $(2,2)$. By \cite[Theorem~3.8]{R}, we may take $n$-planes $\Lambda_1, \Lambda_2 \subset X$ such that $\dim \Lambda_1\cap \Lambda_2=1$. Applying \cite[Lemma~3.10]{R}, we have $\deg \Lambda_1\cdot \Lambda_2=-1$. Thus $\Delta_X$ is not nef. 
\end{proof}

\begin{remark}\label{rem:22} Any $(2n+1)$-dimensional smooth projective complete intersection $X$ of two quadrics has Betti numbers 
\begin{equation}\label{eq:b2}
\nonumber  b_{2k}(X)=1;~~
b_{2k+1}(X)=
\begin{cases}
0 & \text{if $k\neq n$,}\\
2n+2 & \text{if $k=n$}.
\end{cases}
\end{equation}
In particular, any effective cycle on $X$ is nef, $\deg c_{2n+1}(X)=0$ and $X$ is not a fake projective space. Thus we cannot conclude the nefness/non-nefness of the diagonal $\Delta_X$ from the criteria in Section~\ref{subsec:nef}.
\end{remark}

\begin{question}\label{ques:22} Does any odd-dimensional complete intersection $X$ of two quadrics have nef diagonal
?
\end{question}

We do not know the answer to this question even for the $3$-dimensional case (see \cite[Section~8.1]{LO}). 

\begin{corollary}\label{cor:ci:big} Let $X$ be a smooth projective complete intersection of hypersurfaces. Then the diagonal $\Delta_X$ is nef and big if and only if $X$ is a projective space or an odd-dimensional quadric.
\end{corollary}

\begin{proof} Assume that the diagonal is nef and big and $X$ is not a projective space. By Theorem~\ref{them:ci}, $X$ is a quadric or an odd-dimensional complete intersection of two quadrics. If $X$ is a quadric, then the dimension should be odd (see for instance \cite[Section~7.1]{LO}). In the case, $X$ is a fake projective space. If $X$ is an odd-dimensional complete intersection of two quadrics, then the degree of the top Chern class of $X$ is zero. Since $\deg \Delta_X^2=\deg c_{\dim X}(X)=0$, this yields that $\Delta_X$ lies in the boundary of the pseudoeffective cone $\overline{{\rm Eff}}_{\dim X}(X \times X)$. Thus $\Delta_X$ is not big. Hence we see that $X$ is a projective space or an odd-dimensional quadric provided that the diagonal is nef and big. Conversely, it follows from Proposition~\ref{prop:fps} that a projective space and an odd-dimensional quadric have nef and big diagonal. 
\end{proof}

\begin{corollary}\label{cor:CPci} Let $X$ be a complete intersection of hypersurfaces. Assume that the tangent bundle $T_X$ is nef. Then $X$ is a projective space or a quadric. 
\end{corollary}

\begin{proof} Assume that $X$ is neither a projective space nor a quadric. By Theorem~\ref{them:ci}, $X$ is an odd-dimensional complete intersection of two quadrics. We claim that the tangent bundle is not nef in this case. Here we give a sketch of the proof based on the same idea as in \cite{Pan}.

Let $X$ be a $(2n+1)$-dimensional complete intersection of two quadrics. Assume that the tangent bundle is nef. Since $X$ is covered by lines, we may take the family of lines $M$ on $X$ and its universal family $U$. We denote by $p: U \to M$ the universal morphism and by $q: U \to X$ the evaluation morphism with a fiber $F$. By construction, $p$ is a smooth morphism whose fibers are $\P^1$. On the other hand, the nefness of $T_X$ implies the smoothness of $q$. Then it is straightforward to check that $F$ is a $2(n-1)$-dimensional complete intersection of two quadrics. Then we have $$p_M(t)p_{\P^1}(t)=p_X(t)p_F(t),$$ where $p_{Z}(t):=\sum_ib_i(Z)(-t)^i $ is the Poincar\'e polynomial of a variety $Z$. However this contradicts to the fact that $p_{\P^1}(t)=1+t^2, p_X(t)=\sum_{i=0}^{2n+1}t^{2i}-2(n+1)t$ and $p_F(t)=\sum_{i=0}^{2(n-1)}t^{2i}$. 
\end{proof}

\section{Weighted hypersurfaces}

To study del Pezzo varieties of degree one and two in the next section, we compute the degree of the top Chern class of weighted hypersurfaces.  
All results of this section are classically well known, but we include proofs for the reader's convenience. 

For a vector of positive integers $\a=(a_0, a_1, \ldots, a_m)$, let us denote by $\P(\a)$ the weighted projective space of type $\a$. 
Let $X \subset \P(\a)$ be a smooth weighted hypersurface of degree $d$. Assume that $X$ is contained in the smooth locus of $\P(\a)$ and $m \geq 4$. By \cite[Theorem~5.32]{KKA}, $\cO_X(1)$ generates ${\rm Pic}(X)$: ${\rm Pic}(X)=\Z[\cO_X(1)]$. We denote by $h \in A^1(X)$ the class corresponding to $\cO_X(1)$. 

\begin{proposition}[{\cite[Theorem~12.1]{BC}}]\label{prop:genel-Euler} There is an exact sequence of sheaves on $\P(\a)$,
$$
0\to \Omega_{\P(\a)} \to \bigoplus_{i=0}^m\cO_{\P(\a)}(-a_i) \to \cO_{\P(\a)} \to 0.
$$This exact sequence is called {\it the generalized Euler sequence of $\P(\a)$}.
\end{proposition}

\begin{proposition}\label{prop:weighted} Let $X \subset \P(\a)$ be a smooth weighted hypersurface of degree $d$. Assume that $X$ is contained in the smooth locus of $\P(\a)$ and $m \geq 4$. Then the degree of the top Chern class is given by 
$$\deg c_{m-1}(X)=\biggl\{ \sum_{i=0}^{m-1}e_{m-1-i}(a_0, \ldots,a_m)(-d)^i\biggr\} h^{m-1}.
$$
\end{proposition}

\begin{proof}
From the generalized Euler sequence of $\P(\a)$, 
$$
c_t(\Omega_{\P(\a)}|_X)=(1-a_0ht)(1-a_1ht)\cdots (1-a_mht).
$$
By the exact sequence 
$$
0 \to \cO_X(-d) \to \Omega_{\P(\a)}|_X \to \Omega_X \to 0,
$$
we obtain
\begin{eqnarray}
c_t(T_X)
&=&\dfrac{(1+a_0ht)(1+a_1ht)\cdots (1+a_mht)}{1+dht}\nonumber \\
&=&(1+a_0ht)(1+a_1ht)\cdots (1+a_mht) \sum_{j=0}^{m-1} (-dht)^j \nonumber \\
&=&\biggl\{ \sum_{i=0}^{m-1} e_i(a_0, \ldots, a_m)(ht)^i \biggr\} \biggl\{\sum_{j=0}^{m-1} (-dht)^j\biggr\}, \nonumber 
\end{eqnarray}
where $e_i(x_0, \ldots, x_m)$ is the $i$-th elementary symmetric polynomial in $(m+1)$ variables $x_0, \ldots, x_m$:
$$
e_i(x_0, \ldots, x_m)=\sum_{0 \leq j_1< j_2 \ldots<j_i\leq m} x_{j_1}x_{j_2}\cdots x_{j_i}.
$$
As a consequence, our assertion holds.
\end{proof}

\section{Del Pezzo varieties} 

A {\it smooth Fano variety} $X$ is a smooth projective variety with ample anticanonical divisor $-K_X$. For a smooth Fano variety $X$ of dimension $n$, we denote by $i_X$ the {\it Fano index} of $X$, the largest integer by which $-K_X$ is divisible in ${\rm Pic}(X)$. By virtue of \cite{KO}, if the Fano index $i_X$ is at least $n$, then $X$ is isomorphic to a projective space $\P^n$ or a quadric hypersurface $Q^n \subset \P^{n+1}$. Since $\P^n$ and $Q^n$ are homogeneous, in these cases the diagonal is nef. 

Smooth Fano varieties of dimension $n\geq 3$ and index $n-1$ are called {\it smooth del Pezzo varieties}. In this section, we shall classify smooth del Pezzo varieties with nef diagonal. Let us recall the classification of smooth del Pezzo varieties due to T. Fujita and V. A. Iskovskikh: 

\begin{theorem}[{\cite[Theorem~8.11]{Fuj}, \cite{Fuj1, Fuj2, Fuj3}}, \cite{Isk1, Isk2, Isk3}]\label{thm:delpezzo} Let $X$ be a smooth del Pezzo variety of dimension $n \geq 3$ and degree $d=H^n$, where $-K_X=(n-1)H \in {\rm Pic}(X)$. Then $X$ is one of the following.
\begin{enumerate} 
\item If $d=1$, then $X$ is a weighted hypersurface of degree $6$ in the weighted projective space $\P(3, 2, 1, \ldots, 1)$.
\item If $d=2$, then $X$ is a weighted hypersurface of degree $4$ in the weighted projective space $\P(2, 1, \ldots, 1)$. In this case, $X$ is a double cover branched along a quartic in $\P^n$. 
\item If $d=3$, then $X \subset \P^{n+1}$ is a cubic hypersurface.
\item If $d=4$, then $X \subset \P^{n+2}$ is a complete intersection of type $(2,2)$.
\item If $d=5$, then $X$ is a linear section of the Grassmannian $G(2, \C^5) \subset \P^9$ embedded via the Pl\"ucker embedding.
\item If $d=6$, then $X$ is either $\P^1\times \P^1 \times \P^1$, $\P^2 \times \P^2$ or $\P(T_{\P^2})$.
\item If $d=7$, then $X$ is the blow-up of $\P^3$ at a point.
\end{enumerate}
\end{theorem}

By using the above classification result, we shall prove the following.

\begin{theorem}\label{thm:dPnef} Let $X$ be a smooth del Pezzo variety of dimension $n \geq 3$. If $\Delta_X$ is nef, then $X$ is one of the following:
\begin{enumerate}
\item an odd-dimensional complete intersection of two quadrics, 
\item the Grassmannian $G(2, \C^5)$, 
\item a $3$-dimensional linear section of the Grassmannian $G(2, \C^5) \subset \P^9$ embedded via the Pl\"ucker embedding,
\item $\P^1\times \P^1 \times \P^1$,
\item $\P^2 \times \P^2$ or
\item $\P(T_{\P^2})$.
\end{enumerate}
In the cases of $(2), (4), (5), (6)$, $X$ is homogeneous. In the case of $(3)$, $X$ is a fake projective space. 
\end{theorem}

\subsection{Cases of degree one and two}

\begin{proposition}\label{prop:deg1} Let $X$ be a smooth del Pezzo variety of degree one. Then $\Delta_X$ is not nef.
\end{proposition}

\begin{proof} By Theorem~\ref{thm:delpezzo}, $X$ is a weighted hypersurface of degree $6$ in the weighted projective space $\P:=\P(3, 2, 1, \ldots, 1)$. Applying \cite[Proposition~7]{DD}, the singular locus $\P_{\rm sing}$ of $\P$ consists of two points:
$$\P_{\rm sing}=\{(1:0:\ldots: 0), (0:1:0:\ldots: 0)\}. $$ 
Since $\dim X \geq 3$, we have $$\codim_X(X \cap \P_{\rm sing}) \geq 3.$$ Hence $X$ is in general position relative to $\P_{\rm sing}$ in the sense of A. Dimca \cite[Definition~1]{Dim}. Then \cite[Proposition~8]{Dim} tells us that the singular locus of $X$ coincides with $X \cap \P_{\rm sing}$. Since $X$ is smooth by definition, $X$ is contained in the smooth locus of $\P$. 

Moreover $X$ is a double cover of the Veronese cone. This yields that $X$ is a covering space of $\P^n$ of degree $2^n$. Applying Proposition~\ref{prop:weighted}, we obtain
$$\deg c_n(X) = \sum_{i=0}^ne_{n-i}(3,2,1^n)(-6)^i.
$$
Here $e_{n-i}(3,2,1^n)$ means $e_{n-i}(3,2,\underbrace{1, \ldots, 1}_{n{\rm -times}})$. In the following, we use similar notation. 
Remark that
\begin{eqnarray}
 &&e_k(3,2,1^n) \nonumber \\
&=&6e_{k-2}(1^n)+3e_{k-1}(1^n)+2e_{k-1}(1^n)+e_k(1^n)\nonumber \\
&=&6\binom{n}{k-2}+5\binom{n}{k-1}+\binom{n}{k}. \nonumber 
\end{eqnarray}
Thus we have 
\begin{eqnarray}
 \deg c_n(X)&=&\sum_{i=0}^n \biggl\{ 6\binom{n}{i+2}+5\binom{n}{i+1}+\binom{n}{i} \biggr\}(-6)^i \nonumber \\
  &=&\dfrac{1}{6}\sum_{i=0}^{n-2} \binom{n}{i+2}(-6)^{i+2}-\dfrac{5}{6}\sum_{i=0}^{n-1}\binom{n}{i+1}(-6)^{i+1}+\sum_{i=0}^{n} \binom{n}{i}(-6)^i \nonumber \\
&=&\dfrac{1}{6}\biggl\{ (-5)^n+6n-1\biggr\} -\dfrac{5}{6}\biggl\{ (-5)^n-1\biggr\} +(-5)^n \nonumber \\
&=&\dfrac{1}{3}\biggl\{ 3n+2+(-5)^n\biggr\}.  \nonumber 
\end{eqnarray}
If $n$ is odd, then $\deg c_n(X)$ is negative. Thus $\Delta_X$ is not nef. If $n$ is even, then it is straightforward to show the following:
$$\dfrac{1}{3}\biggl\{ 3n+2+(-5)^n\biggr\}>2^n(n+1)~\text{if}~n \geq 4.
$$ 
Hence it follows from Proposition~\ref{prop:LO} that $\Delta_X$ is not nef.
\end{proof}

\begin{proposition}\label{prop:deg2} Let $X$ be a smooth del Pezzo variety of degree two. Then $\Delta_X$ is not nef.
\end{proposition}

\begin{proof} By Theorem~\ref{thm:delpezzo}, $X$ is a weighted hypersurface of degree $4$ in the weighted projective space $\P:=\P(2, 1, \ldots, 1)$. By the same arguments as in Proposition~\ref{prop:deg1}, the singular locus $\P_{\rm sing}$ consists of one point $\{(1: 0: \ldots: 0)\}$ and $X$ is contained in the smooth locus of $\P$.
Moreover, $X$ is a double cover branched along a quartic in $\P^n$. Applying Proposition~\ref{prop:weighted}, we obtain
$$\deg c_n(X) = 2\sum_{i=0}^ne_{n-i}(2,1^{n+1})(-4)^i.
$$
Remark that
$$e_k(2,1^{n+1})=2e_{k-1}(1^{n+1})+e_{k}(1^{n+1})=2\binom{n+1}{k-1}+\binom{n+1}{k}.
$$
Thus we have 
\begin{eqnarray}
 \deg c_n(X)&=&2\sum_{i=0}^n \biggl\{ 2\binom{n+1}{i+2}+\binom{n+1}{i+1} \biggr\}(-4)^i \nonumber \\
  &=&\dfrac{1}{4}\sum_{i=0}^{n-1} \binom{n+1}{i+2}(-4)^{i+2}-\dfrac{1}{2}\sum_{i=0}^{n}\binom{n+1}{i+1}(-4)^{i+1}  \nonumber \\
&=&\dfrac{1}{4}\biggl\{ (-3)^{n+1}-(n+1)(-4)-1\biggr\} -\dfrac{1}{2}\biggl\{ (-3)^{n+1}-1\biggr\} \nonumber \\
&=&\dfrac{1}{4}\biggl\{ 4n+5-(-3)^{n+1}\biggr\}.  \nonumber 
\end{eqnarray}
If $n$ is odd, then $\deg c_n(X)$ is negative. Thus $\Delta_X$ is not nef. If $n$ is even, then it is straightforward to show the following:
$$\dfrac{1}{4}\biggl\{ 4n+5-(-3)^{n+1}\biggr\}>2(n+1)~\text{if}~n \geq 4.
$$ 
Hence it follows from Proposition~\ref{prop:LO} that $\Delta_X$ is not nef.
\end{proof}

\subsection{Cases of degree five}\label{sec:deg5}

Let $X$ be a smooth del Pezzo variety of dimension $n \geq 3$ and degree $5$. By Theorem~\ref{thm:delpezzo}, $X$ is a linear section of the Grassmannian $G(2, \C^5) \subset \P^9$ embedded via the Pl\"ucker embedding. When $n=3$, $X$ is a fake projective space. When $n=6$, $X$ is the Grassmaniann $G(2, \C^5)$. Thus, in these cases, $\Delta_X$ is nef. We study the case of dimension $4$. 

Let $X$ be a del Pezzo $4$-fold of degree $5$. Then $X$ is a linear section of the Grassmannian $G(2, \C^5) \subset \P^9$. 
Let $p_{ij}$ be the Pl\"ucker coordinates of the Grassmannian  $G(2,\C^5) \subset \P({\wedge^k \C^5}^{\vee})$ $(0\leq i < j \leq 4)$. Since any del Pezzo $4$-fold of degree $5$ is isomorphic to each other, we may assume that $X \subset \P((\wedge^k \C^5)^{\vee})$ is defined by 
\begin{eqnarray}
&&p_{01}p_{23}-p_{02}p_{13}+p_{03}p_{12}=0 \nonumber \\
&&p_{01}p_{24}-p_{02}p_{14}+p_{04}p_{12}=0 \nonumber \\
&&p_{01}p_{34}-p_{03}p_{14}+p_{04}p_{13}=0 \nonumber \\
&&p_{02}p_{34}-p_{03}p_{24}+p_{04}p_{23}=0 \nonumber \\
&&p_{12}p_{34}-p_{13}p_{24}+p_{14}p_{23}=0 \nonumber \\
&&p_{12}-p_{03}=0 \nonumber \\
&&p_{13}-p_{24}=0 \nonumber
\end{eqnarray}
Then $X$ contains the following planes:
\begin{itemize}
\item $\Pi:=\{p_{ij}=0\mid (i,j)\neq (0, 1), (0,2), (0,4)\},$
\item $\Lambda:=\{p_{ij}=0\mid (i,j)\neq (0, 1), (0,4), (1,4)\}.$
\end{itemize}
These are the Schubert varieties $\sigma_{3,1}$ and $\sigma_{2,2}$ on $G(2,\C^5)$ respectively \cite[Chapter~1, Section~5]{GH}. Then it is known that $\deg \Pi\cdot \Lambda=-1$ (see for instance \cite[Proof~of~Corollary~4.7]{PZ}). Thus we obtain the following:

\begin{proposition}\label{prop:deg5:4fold} Let $X$ be a smooth del Pezzo $4$-fold of degree $5$. Then $\Delta_X$ is not nef.
\end{proposition}

\begin{proof}[Proof of the Theorem~\ref{thm:dPnef}] If $d=1$ or $2$, then $\Delta_X$ is not nef by Proposition~\ref{prop:deg1} and Proposition~\ref{prop:deg2}. If $d=3$ or $4$, then $X$ is a complete intersection. In these cases, it follows from Theorem~\ref{them:ci} that $\Delta_X$ is not nef provided that $X$ is not an odd-dimensional complete intersection of two quadrics. If $d=6$, then $X$ is homogeneous. Thus $\Delta_X$ is nef. If $d=7$, then $X$ admits a blow-up structure. Hence $\Delta_X$ is not nef by Proposition~\ref{prop:bir}. 

The remaining case is the case of degree five. Let us assume that $d=5$. 
As we have seen in the above Section~\ref{sec:deg5},  $\Delta_X$ is nef provided that $n=3$ or $6$. On the other hand, it follows from Proposition~\ref{prop:deg5:4fold} that $\Delta_X$ is not nef if $n=4$. When $n=5$, we will prove that $\Delta_X$ is not nef in Proposition~\ref{prop:dp:cone}  below.
\end{proof}

We end this section with two corollaries.

\begin{corollary}\label{cor:dp:big} Let $X$ be a smooth complex del Pezzo variety. Then the diagonal $\Delta_X$ is nef and big if and only if $X$ is  a $3$-dimensional linear section of the Grassmannian $G(2, \C^5) \subset \P^9$.
\end{corollary}

\begin{proof} Assume the diagonal $\Delta_X$ is nef and big and $X$ is not a $3$-dimensional linear section of the Grassmannian $G(2, \C^5) \subset \P^9$. By Theorem~\ref{thm:dPnef}, $X$ is an odd-dimensional complete intersection of two quadrics, the Grassmannian $G(2, \C^5)$ or a variety which satisfies $N^1(X) \not\cong \R$. In these cases, it follows from Corollary~\ref{cor:ci:big}, \cite[Section~7.4]{LO} and Proposition~\ref{prop:nefbig:Pic} that the diagonal is not nef and big. On the other hand, a $3$-dimensional linear section of the Grassmannian $G(2, \C^5) \subset \P^9$ is a fake projective space. Thus it has nef and big diagonal by Proposition~\ref{prop:fps}. 
\end{proof}

\begin{corollary}\label{cor:CP} Let $X$ be a smooth del Pezzo variety with nef tangent bundle. Then $X$ is homogeneous.
\end{corollary}

\begin{proof} Assume that $X$ is not homogeneous. Then, by Theorem~\ref{thm:dPnef}, $X$ is isomorphic to an odd-dimensional complete intersection of two quadrics or a $3$-dimensional linear section of the Grassmannian $G(2, \C^5) \subset \P^9$ provided that $X$ is not homogeneous. In these cases, it follows from Corollary~\ref{cor:CPci} and \cite[Theorem~5.1]{CP1} that the tangent bundle is not nef. 
\end{proof}

\section{Spherical varieties}

In this section, we shall study the intersection theory of cycles on spherical varieties. For spherical varieties, we refer the reader to  \cite{Per}. 

\subsection{Spherical varieties}

\begin{definition}\label{def:spherical}
Let $G$ be a reductive linear algebraic group and $B$ a Borel subgroup of $G$. A smooth projective $G$-variety $X$ is {\it ($G$-)spherical} if it has a dense $B$-orbit. A $G$-spherical variety is {\it ($G$-)horospherical} if there is a point $x$ in the open $G$-orbit on $X$ such that the isotropy group $G_x$ contains the unipotent radical of $B$. A spherical $G$-variety $X$ is {\it ($G$-)toroidal} if every $B$-stable divisor but not $G$-stable contains no $G$-orbit. 
\end{definition}

\begin{remark}  We have two remarks.
\begin{enumerate}
\item For a smooth projective $G$-variety $X$, $X$ is spherical if and only if $X$ has finitely many $B$-orbits (see \cite[Theorem~2.1.2]{Per}). 
\item Any smooth projective toric variety is toroidal.
\end{enumerate}
\end{remark}

\begin{proposition}\label{prop:Sph} For a smooth projective spherical variety $X$, the following are equivalent to each other.
\begin{enumerate}
\item $\Delta_X$ is nef. 
\item $\Nef_k(X)=\Eff_k(X)$ for any $k$.
\item The closures of any $B$-orbit on $X$ have non-negative intersection against the closure of every $B$-orbit of the complementary dimension.
\end{enumerate}
\end{proposition}

\begin{proof} $\rm (1) \Rightarrow (2)$ By \cite[Theorem~1.1]{Li}, any spherical variety satisfies $$\Nef_k(X)\subset \Eff_k(X)=\overline{\Eff}_k(X).$$ Thus Proposition~\ref{prop:pseffcone} yields $\Nef_k(X)=\Eff_k(X)$ provided that $\Delta_X$ is nef.\\
$\rm (2) \Rightarrow (1)$ Assume that $\Nef_k(X)=\Eff_k(X)$ for any $k$. By \cite[Corollary~3.5]{Li}, we have $\Nef_k(X \times X)=\Eff_k(X\times X)$ for $0 \leq k \leq 2 \dim X$. In particular, $\Delta_X$ is nef. \\
$\rm (2) \Leftrightarrow (3)$ According to \cite[Theorem~1.3]{FMSS}, the effective cone of a spherical variety is generated by the closures of $B$-orbits. Hence our assertion holds.
\end{proof}

\begin{corollary}[{\cite[Theorem~1.2]{Li}}]\label{cor:Li} For a smooth projective toroidal variety $X$ of dimension $n$, the following are equivalent to each other. 
\begin{enumerate}
\item $\Delta_X$ is nef. 
\item $\Nef_k(X)=\Eff_k(X)$ for any $k$.
\item $\Nef_k(X)=\Eff_k(X)$ for some $1 \leq k \leq n-1$.
\item $X$ is rational homogeneous.
\end{enumerate}
\end{corollary}

\begin{proof} The equivalence of $\rm (3)$ and $\rm (4)$ is due to Q. Li {\cite[Theorem~1.2]{Li}}. The remaining part directly follows from Proposition~\ref{prop:Sph}.
\end{proof}

\begin{corollary}\label{cor:spherical:nefbig} For a smooth projective spherical variety $X$, assume that $\Delta_X$ is nef and big. Then $N_k(X) \cong \R$ for any $k$.
\end{corollary}

\begin{proof} Let us assume that $\Delta_X$ is nef and big. As in the proof of Proposition~\ref{prop:Sph}, we have $\Nef_{\dim X}(X \times X)=\Eff_{\dim X}(X\times X)$. Then it follows from \cite[Example~4.5]{FL2} that $\Delta_X$ is universally pseudoeffective. Applying \cite[Corollary~4.2]{LO}, we see that $N_k(X) \cong \R$ for any $k$. 
\end{proof}

\subsection{Odd symplectic Grassmannians}

In this section, we recall some known results on odd symplectic Grassmannians. 
All of materials here are in \cite[Section 3 and 4]{Mih}.

Let $E$ be a $(2n+1)$-dimensional complex vector space and $\omega \in \wedge^2 E^{\vee}$ a skew symmetric bilinear form of rank $2n$. We call this $\omega$ an {\it odd symplectic form on $E$}. In \cite{Proc}, R. A. Proctor introduces the {\it odd symplectic group}: 
$$
\Sp({2n+1}):=\{g\in {\rm  GL}(E) \mid \omega(gu,gv)=\omega(u,v) ~\text{for~all}~ u, v \in E\}. 
$$
Let $\overline{J}$ be the matrix of the form
$$\left(
    \begin{array}{cc}
      O & J \\
      -J & O 
    \end{array}
  \right),  
  $$
where $J$ be the anti-diagonal $(1,1, \ldots, 1)$ of size $n \times n$. 
Taking a basis $\{\bm{e_0}, \bm{e_1}, \ldots, \bm{e_{2n}}\}$ of $E$ suitably, we may assume that the form $\omega$ is given by the $2n \times 2n$ matrix
$$
(\omega(\bm{e_i},\bm{e_j}))=
\left(
    \begin{array}{cc}
      0 & 0 \\
      0 & \overline{J}
    \end{array}
  \right).
  $$
The restriction of $\omega$ to the subspace $F=\langle \bm{e_1}, \ldots, \bm{e_{2n}} \rangle_{\C}$ is a usual symplectic form.
Then the odd symplectic group $\Sp({2n+1})$ is represented as  
$$
\Sp({2n+1})=\biggl\{ \left(
    \begin{array}{cc}
      \lambda & \ell \\
      0 & S 
    \end{array}
  \right)  
  \mid \lambda \in \C^{\ast}, \ell \in \C^{2n}, S \in \Sp(^{2n})
 \biggr\},
$$ where $\Sp({2n})$ is the usual symplectic group with respect to the form $\omega|_{F}$.
This description tells us that $\Sp(2n+1)$ admits a Levi decomposition:
$$ 
\Sp(2n+1)\cong (\C^{\ast} \times \Sp({2n})) \ltimes \C^{2n}.
$$
The subgroup $B$ of $\Sp({2n+1})$ of upper triangular matrices is a Borel subgroup.

The {\it odd symplectic Grassmannian} is defined as a subvariety of a Grassmannian $G(k, E)$ parametrizing isotropic subspaces with respect to $\omega$: 
$$G_{\omega}(k, E):=\{[V] \in G(k, E)\mid \omega(V, V)=0\}.
$$ 
According to \cite[Proposition~4.3]{Mih}, the odd symplectic Grassmannian $G_{\omega}(k, E)$ has an action of the odd symplectic group $\Sp({2n+1})$ with two orbits
$$
X_0:=\{V \in G_{\omega}(k, E) \mid \bm{e_0} \in V\}, 
X_1:=\{V \in G_{\omega}(k, E) \mid \bm{e_0} \not\in V\}. 
$$
Furthermore, the closed orbit $X_0$ is isomorphic to $G_{\omega}(k-1, F)$, and the open orbit $X_1$ is isomorphic to the total space of the dual of the tautological bundle of $G_{\omega}(k, F)$. From this description, we see that $G_{\omega}(k, E)$ is a horospherical variety.
By \cite[Proposition~4.1]{Mih}, $G_{\omega}(k, E)$ is a smooth projective subvariety of codimension $\frac{k(k-1)}{2}$ of the Grassmannian $G(k, E)$. More precisely, it is the zero locus of a general section on $\wedge^2 \cS$, where $\cS$ is the rank $k$ tautological subbundle on $G(k, E)$ (see the proof of \cite[Proposition~4.1]{Mih}).
In particular, $G_{\omega}(2, E)$ is a hyperplane section of $G(2, E) \subset \P(\wedge^2 E^{\vee})$ by $\{\omega=0\}$. 

Since the odd symplectic form $\omega$ can be extended to a (usual) symplectic form $\tilde{\omega}$ on $\tilde{E}=\C^{2n+2}$ (see \cite[Section~3.2]{Mih}), $G_{\omega}(k, E)$ is a subvariety of a (usual) symplectic Grassmannian 
$$G_{\tilde{\omega}}(k, \tilde{E}):=\{[V] \in G(k, \tilde{E})\mid \tilde{\omega}(V, V)=0\}.$$ 
Moreover $G_{\omega}(k, E)$ identifies with the Schubert variety of $G_{\tilde{\omega}}(k, \tilde{E})$ that parametrizes $k$-dimensional isotropic subspaces contained in the hyperplane $E \subset \tilde{E}$. Then we may define {\it Schubert varieties} of the odd symplectic Grassmannian $G_{\omega}(k, E)$ to be those of $G_{\tilde{\omega}}(k, \tilde{E})$ contained in $G_{\omega}(k, E)$ (see \cite[Section~4.8]{Mih}).  Then we have

\begin{proposition}[{\cite[Proposition~4.12]{Mih}}]\label{prop:Mih} For an odd symplectic Grassmannian $G_{\omega}(k, \C^{2n+1})$, Schubert varieties coincide with the closures of $B$-orbits.
\end{proposition}

\subsection{Cones of cycles of the odd symplectic Grassmannian of lines}

Combining Proposition~\ref{prop:Mih} and \cite[Theorem~1.3]{FMSS}, we obtain the following:

\begin{proposition}\label{prop:OSG:effective} For an odd symplectic Grassmannian $G_{\omega}(k, \C^{2n+1})$, the effective cone of cycles is generated by Schubert varieties.
\end{proposition}

From now on, we focus on the odd symplectic Grassmannian of lines $G_{\omega}(2, \C^{2n+1})$. 
To describe Schubert varieties of odd symplectic Grassmannians, we employ the notation by C. Pech \cite[P. 191]{Pe}.
As in \cite[Section~1.2]{Pe}, we denote by $\tau_{\lambda}$ the cohomology class associated to the Schubert variety $X_{\lambda}(F_{\bullet}) \subset G_{\omega}(2, \C^{2n+1})$, where $F_{\bullet}$ is an isotropic complete flag of $\C^{2n+1}$ and $\lambda=(\lambda_1, \lambda_2)$ (see for details \cite[Section~1.1]{Pe}). The codimension of the cycle $\tau_{\lambda}$ is equal to $\lambda_1+\lambda_2$. Then, by \cite[Remark 1.1.1]{Pe}, $\lambda$ is either 
\begin{itemize}
\item $(n-2)$-strict partitions $(2n-1 \geq \lambda_1 \geq \lambda_2 \geq 0)$, or 
\item the partition $(2n-1, -1)$.
\end{itemize}

Assume that $n \geq 2$. We chose non-negative integers $a \geq b$ satisfying $a+b=2n-1$.
From the Pech's Pieri formula \cite[Proposition~5]{Pe}, the intersection number of the cycles $\tau_{a,b}$ and $\tau_{2n-1, -1}$ is given by 
$$
\deg(\tau_{a,b}\cdot \tau_{2n-1, -1}) =(-1)^{a-1}.
$$ 
This implies that $\tau_{2n-1, -1}$ is a non-nef effective cycle. Thus we have

\begin{proposition}\label{prop:OSGL} Let $X$ be the odd symplectic Grassmannian of lines $G_{\omega}(2, \C^{2n+1})$. Then $\Delta_X$ is not nef. 
\end{proposition}

Furthermore, we may describe the cone of pseudoeffective/nef cycles of a del Pezzo $5$-fold of degree $5$ as follows. Let $X$ be a del Pezzo $5$-fold of degree $5$. By Theorem~\ref{thm:delpezzo}, $X$ is a hyperplane section of $G(2, \C^5)$. Thus it is an odd symplectic Grassmaniann $G_{\omega}(2, \C^{5})$. According to Proposition~\ref{prop:OSG:effective} and the above argument, the cone of effective cycles on $X$ is generated by 
\begin{itemize}
\item $\tau_{(0,0)}$: the whole space $X$
\item $\tau_{(1,0)}$: a hyperplane section of $X$
\item $\tau_{(2,0)}$: a cycle of codimension $2$
\item $\tau_{(3,-1)}$: a cycle of codimension $2$
\item $\tau_{(3,0)}$: a plane
\item $\tau_{(2,1)}$: a plane
\item $\tau_{(3,1)}$: a line
\item $\tau_{(3,2)}$: a point.
\end{itemize}
Thanks to \cite[Proposition~4~and~5]{Pe}, the intersection numbers are given by
$$
\deg(\tau_{(1,0)}\cdot \tau_{(3,1)})=1, 
\deg(\tau_{(2,0)}\cdot \tau_{(3,0)})=0, \deg(\tau_{(2,0)}\cdot \tau_{(2,1)})=1,$$
$$\deg( \tau_{(3,-1)}\cdot \tau_{(3,0)})=1, \deg(\tau_{(3,-1)}\cdot \tau_{(2,1)})=-1.
$$
As a consequence, we obtain the following.

\begin{proposition}\label{prop:dp:cone} Let $X$ be a smooth del Pezzo $5$-fold of degree five. Then we have 
\begin{align}
\Nef^2(X)&=\R_{\geq 0} \tau_{(2,0)} \oplus \R_{\geq 0} [\tau_{(2,0)}+\tau_{(3,-1)}], \nonumber \\
\cEff^2(X)&=\R_{\geq 0} \tau_{(2,0)} \oplus \R_{\geq 0} \tau_{(3,-1)},   \nonumber  \\
\Nef^3(X)&=\R_{\geq 0} \tau_{(3,0)} \oplus \R_{\geq 0} [\tau_{(3,0)}+\tau_{(2,1)}], \nonumber \\
\cEff^3(X)&=\R_{\geq 0} \tau_{(3,0)} \oplus \R_{\geq 0} \tau_{(2,1)}.   \nonumber  
\end{align}
In particular, $\Delta_X$ is not nef. 
\end{proposition}

\begin{remark}[An answer to a question by Q. Li]\label{rem:Li}
A del Pezzo $5$-fold $X$ of degree $5$ is a smooth projective horospherical variety satisfying $\Nef^1(X)=\overline{\Eff}^1(X)$ and $\Nef^2(X) \neq \overline{\Eff}^2(X)$. This gives an answer to a question raised by Q. Li \cite[pp. 971-972]{Li}.
\end{remark}

\bibliographystyle{plain}
\bibliography{NefDiagonal}

\end{document}